\documentclass[reqno,11pt]{amsart}
\usepackage{amssymb, amsmath}
\usepackage{amsthm, amsfonts,mathrsfs}
\usepackage[T1]{fontenc}
\usepackage{times}
\usepackage{graphicx}
\usepackage{epsfig}
\usepackage{color}
\usepackage[all]{xypic}
\newtheorem{theorem}{Theorem}[section]

\newtheorem{definition}{Definition}[section]
\newtheorem{lemma}{Lemma}[section]
\newtheorem{proposition}{Proposition}[section]
\newtheorem{remark}{Remark}[section]

\numberwithin{equation}{section}

\newcommand{\mr}{\mathbb{R}}

\newcommand{\p}{\partial}
\newcommand{\N}{{\mathbb N}}

\newcommand{\e}{\varepsilon}
\newcommand{\n}{\nu}

\newcommand{\LC}{\left(}
\newcommand{\RC}{\right)}
\newcommand{\LB}{\left[}
\newcommand{\RB}{\right]}
\newcommand{\LCB}{\left\{}
\newcommand{\RCB}{\right\}}

\newcommand{\normmm}[1]{{\left\vert\kern-0.25ex\left\vert\kern-0.25ex\left\vert #1\right\vert\kern-0.25ex\right\vert\kern-0.25ex\right\vert}}
\allowdisplaybreaks[4]

\begin{document}
\title[fractional Camassa-Holm equation]
{The Cauchy problem for fractional Camassa-Holm equation in Besov space}%
\author[Lili Fan]{Lili Fan}%
\address[Lili Fan]{College of Mathematics and Information Science,
Henan Normal University, Xinxiang 453007, China}
\email{fanlily89@126.com}
\author[Hongjun Gao]{Hongjun Gao}
\address[Hongjun Gao]{School of Mathematical Sciences, Institute of Mathematics,
 Nanjing Normal University, Nanjing 210023, China}
\email{gaohj@hotmail.com}
\author[Junfang Wang ]{Junfang Wang}%
\address[Junfang Wang]{School of Mathematics and statistics, North China University of Water Resources and Electric Power, Zhengzhou, Henan 450045,   China}
\email{wangjunfang18@ncwu.edu.cn}
\author[Wei Yan $^{\dag}$]{Wei Yan$^{\dag}$}%
\address[Wei Yan]{College of Mathematics and Information Science,
Henan Normal University, Xinxiang 453007, China}
\email{011133@htu.edu.cn\, (Corresponding author)}


\begin{abstract}In this paper, we consider the fractional Camassa-Holm equation modelling the propagation of small-but-finite amplitude long unidirectional waves in a nonlocally and nonlinearly elastic medium. First, we establish the local well-posedness in Besov space $B^{s_0}_{2,1}$ with $s_0=2\nu-\frac 1 2$ for $\nu>\frac 3 2 $ and $s_0=\frac 5 2$ for $1<\nu\leq \frac 3 2 $. Then, with a given analytic initial data, we establish the analyticity of the solutions in both variables, globally in space and locally in time.
\end{abstract}

\date{}

\maketitle




\section{Introduction}
\large
This paper is concerned with an evolution equation, named fractional Camassa-Holm (fCH) equation, which models the propagation of small-but-finite amplitude, long unidirectional waves in a one-dimensional infinite, homogeneous medium made of nonlocally and nonlinearly elastic material \cite{EEE}
\begin{align}\label{1.1}
 u_t&+u_x+uu_x+\frac {3} {4}(-\p_x^2)^\n u_{x}+\frac {5} {4} (-\p_x^2)^\n u_{t}\nonumber\\
 &+\frac {1} 4\left[2(-\p_x^2)^\n (uu_x)+u(-\p_x^2)^\n u_{x}\right]=0,\qquad x\in\mr,\quad t>0,
\end{align}
where $\n\geq 1$ is a constant which may not be an integer. It is remarkable that when $\n=1$, \eqref{1.1} reduces to the following classical Camassa-Holm (CH) equation
\begin{equation}\label{1.2}
u_t+k_1(u_x-u_{xxx})+3uu_x-u_{xxt}=k_2\left(2u_xu_{xx}+uu_{xxx}\right).
\end{equation}
This prominent CH equation, for which the ratio of the nonlinear terms of \eqref{1.2} being $3:2:1$, was first derived formally by Fuchssteiner and Fokas \cite{FF} as a bi-Hamiltonian equation and later derived in the context of water waves as a model for unidirectional propagation of shallow water waves of moderate amplitude by Camassa and Holm \cite{CH} (see also the alternative derivation in \cite{Co,Co5,J}). The CH equation has been studied extensively in the last twenty years because of its many remarkable properties: infinitely many conservation laws and complete integrability \cite{CH,FS}, existence of peaked solitons and multi-peakons \cite{ACH,CH}, well-posedness and breaking waves \cite{Br1,Br2,Co1,Co2,Co3,Co4,RD1,RD2,HMPZ}, just to mention a few.

Recent years, increasing attention has been paid to the fractional equations. For instance, the fractional Korteweg-de Vries (fKdV) equation and the fractional Benjamin-Bona-Mahony (fBBM) equation have been obtained and studied in \cite{EEE,J0,P}, the Camassa-Holm equations with fractional dissipation and the Camassa-Holm equations with fractional laplacian viscosity have been investigated in \cite{GHM,GL}. Concerning the fCH equation \eqref{1.1}, the local well-posedness for initial data $u_0 \in H^{s}(\mr)$, $s>\frac{5}{2}$ has been established by employing a semigroup approach due to Kato \cite{Ka1} in \cite{Mu}, which, to our knowledge, is the only result on the Cauchy problem for the fCH equation.

In this paper, we refine the corresponding result in \cite{Mu} by investigating the local well-posedness of the Cauchy problem for \eqref{1.1} in Besov space $B^{s_0}_{2,1}$ with $s_0=2\nu-\frac 1 2$ for $\nu>\frac 3 2 $ and $s_0=\frac 5 2$ for $1<\nu\leq \frac 3 2 $. By virtue of the Littlewood-Paley decomposition, nonhomogeneous Besov spaces and iterative method, the methods proposed in \cite{RD1,RD2,RD3} have been applied with success when studying the well-posedness of various shallow water wave equations in Besov space (see for example \cite{FG,FY,Liu,HT,Yin,MY,ZY,YYY}). To obtain our result, we felicitously recast the equation \eqref{1.1} in a form of nonlocal conservation law
\begin{equation}\label{3.1}
  u_t+\frac 3 5(1+u)u_x=-\LC1+\frac 5 4(-\p_x^2)^\n\RC^{-1}\LC\frac 2 5 u_x+ \frac 2 5 uu_x+\frac 1 4[u,(-\p_x^2)^\n]u_x\RC,
\end{equation}
where $[,]$ denotes the usual commutator of the linear operators. Then another difficulty arises as the appearance of commutator estimates of $[u,(-\p_x^2)^\n]u_x$ in $B^{s_0}_{2,1}$ and $B^{s_0-1}_{2,\infty}$. By the Bony decomposition, we break up the commutator term into the paraproduct terms and the remainder terms, which is beneficial to employing their continuity properties, and hence the desired estimates are obtained as presented in Lemma \ref{com-e}. Then we can establish a uniform bound for the approximate solutions on a sufficiently small time-interval by employing the mean value theorem for integrals. This opens the path to obtain convergence and thus existence and uniqueness and continuous dependence on the initial datum are treated in separate steps of the proof.

Furthermore, provided that the initial profile $u_0$ is an analytic function on the real line $\mr$, we obtain the analyticity of the corresponding solutions in both variables, with $x\in\mr$ and $t$ in an interval around zero. Analyticity is inherent to travelling water waves (see \cite{CE}).

We supplement \eqref{1.1} with the initial data
\begin{equation}\label{1.3}
u(x,0)=u_0(x),\quad x\in \mathbb{R}.
\end{equation}
To introduce the main results, we define
\[
E^{s_0}_{2,1}(T)=C([0,T];B^{s_0}_{2,1})\cap C^1([0,T];B^{s_0-1}_{2,1}).
\]
The main results of this paper are as follows:

\begin{theorem}\label{the1.1}
Let $u_0 \in B^{s_0}_{2,1}$ with $s_0=2\nu-\frac 1 2$ for $\nu>\frac 3 2 $ and $s_0=\frac 5 2$ for $1<\nu\leq \frac 3 2 $. There exists a time $T>0$ such that the problem \eqref{1.1} and \eqref{1.3} has a unique solution in $E^{s_0}_{2,1}(T)$. Moreover, the solution depends continuously on the initial data, i.e., the mapping $\Phi:u_0\mapsto u$ is continuous from a neighborhood of $u_0 \in B^{s_0}_{2,1}$ into $E^{s_0}_{2,1}(T)$.
\end{theorem}

\begin{remark}\label{rem1.1} We obtain the local well-posedness of equations \eqref{1.1} and \eqref{1.3} in the case $B^{s_0}_{2,1}$. However, this is not true in the case $B^{s_0}_{2,\infty}$ in view of the proof of Proposition 4 in \cite{RD2}. Noting that $B^{s_0}_{2,1}\hookrightarrow H^{s_0} \hookrightarrow B^{s_0}_{2,\infty}$, one can see that $s=s_0$ is the critical index.
\end{remark}

Referring to the definition of the space $E_{s}$ in \eqref{5.1}, we present the following analytic result.
\begin{theorem}\label{the1.2}
If the initial data $u_0$ is real analytic on the line $\mr$ and belongs to a space $E_{s}$ for some $0<s\leq 1$, then there exist an $\e>0$ and a unique solution $u$ to the problem \eqref{1.1} and \eqref{1.3} that is analytic on $\mr\times[0,\e)$.
\end{theorem}
\begin{remark}\label{rem1.2}
A rereading of the proof of Theorem \ref{the1.1} and Theorem \ref{the1.2} yields that there exists a real analytic extension of $u$ to $(-\e,\e)$.
\end{remark}

In the sequel, we will, for notational convenience, deal with the following initial value problem with different coefficients, which implies the Theorem \ref{the1.1}-Theorem\ref{the1.2} as the concrete values of the coefficients have no impact on the results.
\begin{equation}\label{3.2}
\begin{cases}
u_t+(1+u)u_x=\p_x P(D)f_1(u)+P(D)f_2(u,u_x),\\
u(0,x)=u_0(x),\\
\end{cases}
\end{equation}
with the operator $P(D)=-\LC1+(-\p_x^2)^\n\RC^{-1}$, and
\begin{align}\label{3.3}
f_1(u)=u+u^2,\quad\;\;\; f_2(u,u_x)=[u,(-\p_x^2)^\n]u_x.
\end{align}

The rest of this paper is organized as follows. In Section 2, we give some preliminaries.
Section 3 concentrates on the primary commutator estimates needed in the proof of Theorem \ref{the1.1} and Theorem \ref{the1.2}. Section 4 aims at proving the local well-posedness for the Cauchy problem \eqref{1.1} and \eqref{1.3} in Besov space $B^{s_0}_{2,1}$. Section 5 is devoted to studying the analyticity of the Cauchy problem \eqref{1.1} and \eqref{1.3} based on a contraction type argument in a suitably chosen scale of the Banach spaces.
\section{preliminaries}

For convenience of the reader, we recall some conclusions on the properties of Littlewood-Paley decomposition, the nonhomogeneous Besov spaces and the theory of the transport equation. One may check \cite{RD4,RD1,RD2,RD3,MWZ} for more details.

\begin{lemma}\label{lem2.1}(Littlewood-Paley decomposition). There exist two smooth radial functions $(\chi,\phi)$ valued in $[0,1]$, such that $\chi$ is supported in the ball $B=\{\xi \in \mathbb{R}^n, |\xi|\leq \frac{4}{3}\}$ and $\phi$ is supported in the ring $C=\{\xi \in \mathbb{R}^n, \frac{4}{3}\leq |\xi| \leq \frac{8}{3}\}$. Moreover,
\[
\forall \xi \in \mathbb{R}^n,\quad \chi(\xi)+\sum_{q\geq 0}\phi(2^{-q}\xi)=1
\]
and
\[
\textrm{Supp}\;\phi(2^{-q}\cdot)\cap \textrm{Supp}\;\phi(2^{-q'}\cdot)=\emptyset,\quad if\;|q-q'|\geq2,
\]

\[
\textrm{Supp}\;\chi(\cdot)\cap \textrm{Supp}\;\phi(2^{-q}\cdot)=\emptyset,\quad if\;|q|\geq1.
\]
Then for $u\in \mathcal{S'}(\mathbb{R}^n)$, the nonhomogeneous dyadic operators are defined as follows:
\[
\aligned
\triangle_qu&=0,\quad if\;q\leq-2,\\
\triangle_{-1}u&=\chi(\mathcal{D})u=\mathscr{F}^{-1}_x\chi\mathscr{F}_xu,\\
\triangle_{q}u&=\phi(2^{-q}\mathcal{D})u=\mathscr{F}^{-1}_x\phi(2^{-q}\xi)\mathscr{F}_xu,\quad if q\geq 0.
\endaligned
\]
Thus $u=\sum_{q\geq0}\triangle_qu$ in $\mathcal{S'}(\mathbb{R}^n)$.
\end{lemma}
\begin{remark}\label{rem2.1} The low frequency cut-off $S_q$ is defined by
\[
S_qu=\sum_{p=-1}^{q-1}\triangle_pu=\chi(2^{-q}\mathcal{D})u=\mathscr{F}^{-1}_x\chi(2^{-q}\xi)
\mathscr{F}_xu,\quad \forall q \in N.
\]
It is easily checked that
\[
\aligned
\triangle_p\triangle_q&\equiv0,\quad |p-q|\geq2,\\
\triangle_q(S_{p-1}u\triangle_pv)&\equiv0,\quad |p-q|\geq5,\;\;\forall u,v\in \mathcal{S'}(\mathbb{R}^n)
\endaligned
\]
as well as
\[
\|\triangle_qu\|_{L^p}\leq \|u\|_{L^p},\qquad \|S_qu\|_{L^p}\leq C\|u\|_{L^p},\quad \forall 1\leq p\leq +\infty
\]
with the aid of Young's inequality, where C is a positive constant independent of $q$.
\end{remark}
\begin{definition}\label{def2.1} (\emph{Besov spaces}). Let $s\in \mathbb{R}$, $1\leq p\leq +\infty$. The nonhomogeneous Besov space $B^s_{p,r}(\mathbb{R}^n)$ is defined by
\[
B^s_{p,r}(\mathbb{R}^n)=\{f\in \mathcal{S'}(\mathbb{R}^n):\|f\|_{B^s_{p,r}}=\|2^{qs}\triangle_{q}f\|_{l^r(L^P)}
=\|(2^{qs}\|\triangle_qf\|_{L^p})_{q \geq -1}\|_{l^r}<\infty \}.
\]
In particular, $B^{\infty}_{p,r}=\bigcap_{s\in \mathbb{R}}B^s_{p,r}$.
\end{definition}
\begin{lemma}\label{lem2.2}Let $s\in \mathbb{R}$, $1\leq p,r,p_{j},r_{j} \leq \infty$, $j=1,2$, then:

\noindent(1) Topological properties: $B^s_{p,r}$ is a Banach space which is continuously embedded in $ \mathcal{S'}$.

\noindent(2) Density: $C^{\infty}_{c}$ is dense in $B^s_{p,r} \Leftrightarrow 1\leq p,r < \infty$.

\noindent(3) Embedding: $B^s_{p,r}\hookrightarrow B^{\tilde{s}}_{p,\tilde{r}}$, if $\tilde{s}<s$ or $\tilde{s}=s$ and $\tilde{r}\geq r$. $B^s_{p_1,r_1}\hookrightarrow B^{s-n(\frac{1}{p_1}-\frac{1}{p_2})}_{p_2,r_2}$, if $ p_1\leq p_2$ and $ r_1\leq r_2$ and $B^0_{p,1}\hookrightarrow L^p \hookrightarrow B^{0}_{p,\infty}$.

\noindent(4) Algebraic properties: $\forall s>0$, $ B^s_{p,r}\cap L^{\infty} $ is a Banach algebra. $ B^s_{p,r}$ is a Banach algebra $\Leftrightarrow B^s_{p,r}\hookrightarrow L^{\infty} \Leftrightarrow s>\frac{1}{p}$ or ( $s \geq \frac{1}{p}$ and $r=1$ ). In particular, $B^{1/2}_{2,1}$ is continuously embedded in $B^{1/2}_{2,\infty} \cap L^{\infty}$ and $ B^{1/2}_{2,\infty} \cap L^{\infty} $ is a Banach algebra.

\noindent(5) 1-D Moser-type estimates:

(i) For $s>0$,
\[
\|fg\|_{B^s_{p,r}} \leq C(\|f\|_{B^s_{p,r}}\|g\|_{L^{\infty}}+\|g\|_{B^s_{p,r}}\|f\|_{L^{\infty}}).
\]

(ii) $\forall s_1\leq \frac{1}{p}< s_2$ ( $s_2\geq \frac{1}{p}$ if $r=1$ ) and $s_1+s_2>0$, we have
\[
\|fg\|_{B^{s_1}_{p,r}} \leq C\|f\|_{B^{s_1}_{p,r}}\|g\|_{B^{s_2}_{p,r}}.
\]

\noindent(6) Complex interpolation:
\[
\|f\|_{B^{\theta s_1+(1-\theta)s_2}_{p,r}}\leq \|f\|^{\theta}_{B^{s_1}_{p,r}}\|f\|^{1-\theta}_{B^{s_2}_{p,r}},\quad \forall f\in B^{s_1}_{p,r}\cap B^{s_2}_{p,r},\quad \forall \theta \in [0,1].
\]

\noindent(7) Real interpolation: $ \forall \theta \in (0,1), s_1<s_2, s=\theta s_1+(1-\theta)s_2$, there exists a constant $C$ such that
\[
\|u\|_{B^s_{p,1}} \leq \frac{C(\theta)}{s_2-s_1}\|u\|^{\theta}_{B^{s_1}_{p,\infty}}
\|u\|^{1-\theta}_{B^{s_2}_{p,\infty}}, \forall u\in B^{s_1}_{p,\infty}.
\]

In particular, for any $0<\theta<1$, we have
\begin{equation}\label{2.1}
\|u\|_{B^{s_0-1}_{2,1}} \leq \|u\|_{B^{s_0-\theta}_{2,1}}\leq
C(\theta) \|u\|^{\theta}_{B^{s_0-1}_{2,\infty}}
\|u\|^{1-\theta}_{B^{s_0}_{2,\infty}}.
\end{equation}

\noindent(8) Fatou lemma: if $(u_n)_{n\in \mathbb{N}} $is bounded in $B^s_{p,r}$ and $u_n\rightarrow u$ in $\mathcal{S'}$, then $u \in B^s_{p,r}$ and
\[
\|u \|_{ B^s_{p,r}} \leq \liminf_{n\rightarrow \infty}\|u_{n}\|_{ B^s_{p,r}}.
\]

\noindent(9) Let $m\in \mathbb{R}$ and $f$ be an $s^m$ -multiplier (i.e., $f: \mathbb{R}^n\rightarrow \mathbb{R}$ is smooth and satisfies that $\forall \alpha \in N^n$, $\exists$ a constant $C_{\alpha}$, s.t. $|\partial_{\alpha} f(\xi)|\leq C_{\alpha} (1+|\xi|)^{m-|\alpha|}$ for all $\xi \in \mathbb{R}^n)$. Then the operator $f(D)$ is continuous from $B^s_{p,r}$ to $B^{s-m}_{p,r}$.

\noindent(10) The paraproduct is continuous from $B^{-1/p}_{p,1}\times (B^{1/p}_{p,\infty} \cap L^{\infty}) $ to $B^{-1/p}_{p,1}$, i.e.,
\[
\|fg\|_{B^{-1/p}_{p,\infty}}\leq C \|f\|_{B^{-1/p}_{p,1}}\|g\|_{B^{1/p}_{p,\infty} \cap L^{\infty}}.
\]

\noindent(11) A logarithmic interpolation inequality
\[
\|f\|_{B^{s}_{p,1}}\leq C \|f\|_{B^{s}_{p,\infty}}
\ln(e+\frac{\|f\|_{B^{s+1}_{p,\infty}}}{\|f\|_{B^{s}_{p,\infty}}}).
\]
\end{lemma}

\begin{lemma}\label{lem2.3}
Let $1\leq p,r\leq \infty$ and $s>-\min(\frac{1}{p},1-\frac{1}{p})$. Assume that $f_0\in B^s_{p,r}$, $F\in L^1(0,T;B^s_{p,r})$ and $\partial_xv$ belongs to $L^1(0,T;B^{s-1}_{p,r})$ if $s>1+\frac{1}{p}$ or to $L^1(0,T;B^{1/p}_{p,r}\cap L^{\infty})$ otherwise. If $f\in L^{\infty}(0,T;B^s_{p,r})\cap C([0,T];\mathcal{S'}(\mathbb{R}))$ solves the following 1-D linear transport equation:
\begin{equation}\label{T}
 \left\{\begin{array}{l}
f_t+vf_x=F,\\
f|_{t=0}=f_0.
\end{array}\right. \tag{T}
\end{equation}
then there exists a constant $C$ depending only on $s,p,r$ such that the following statements hold:

\noindent(1) If $r=1$ or $s\neq 1+\frac{1}{p}$, then
\[
\|f\|_{B^s_{p,r}}\leq \|f_0\|_{B^s_{p,r}}+\int^t_0\|F(\tau)\|_{B^s_{p,r}}d\tau
+C\int^t_0V'(\tau)\|f(\tau)\|_{B^s_{p,r}}d\tau,
\]
or
\begin{equation}\label{2.2}
\|f\|_{B^s_{p,r}}\leq e^{CV(t)}(\|f_0\|_{B^s_{p,r}}+\int^t_0e^{-CV(\tau)}\|F(\tau)\|_{B^s_{p,r}}d\tau)
\end{equation}
holds, where $ V(t)=\int^t_0\|v_x(\tau)\|_{B^{1/p}_{p,r}\cap L^{\infty}}d\tau$ if $s<1+\frac{1}{p}$ and $ V(t)=\int^t_0\|v_x(\tau)\|_{B^{s-1}_{p,r}}d\tau$ else.

\noindent(2) If $s\leq 1+\frac{1}{p}$, $f'_0\in L^{\infty}$, $f_x \in L^{\infty}((0,t)\times \mathbb{R})$ and $F_x \in L^{1}(0,T;L^{\infty})$, then
\begin{align}
\|&f(t)\|_{B^s_{p,r}}+\|f_x(t)\|_{L^{\infty}}\nonumber \\
&\leq e^{CV(t)}(\|f_0\|_{B^s_{p,r}}+\|f_{0x}\|_{L^{\infty}}
+\int^t_0e^{-CV(\tau)}[\|F(\tau)\|_{B^s_{p,r}}+\|F_x(\tau)\|_{L^{\infty}}]d\tau)\nonumber
\end{align}

with $ V(t)=\int^t_0\|v_x(\tau)\|_{B^{1/p}_{p,r}\cap L^{\infty}}d\tau$.

\noindent(3) If $f=v$, then for all $s>0$, the estimate \eqref{2.2} holds with $ V(t)=\int^t_0\|v_x(\tau)\|_{L^{\infty}}d\tau$.

\noindent(4) If $r<\infty$, then $f\in C([0,T];B^s_{p,r})$. If $r=\infty$, then $f\in C([0,T];B^{s'}_{p,r})$ for all $s'<s$.
\end{lemma}

\begin{lemma}\label{lem2.4}(Existence and uniqueness) Let $p,r,s,f_0$ and $F$ be as in the statement of Lemma \ref{lem2.3}. Assume that $v\in L^{\rho}(0,t;B^{-M}_{\infty,\infty})$ for some $\rho>1$ and $M>0$ and $v_x\in L^{1}(0,T;B^{s-1}_{p,r})$ if $s>1+\frac{1}{p}$ or $s=1+\frac{1}{p}$ and $r=1$ and $v_x\in L^{1}(0,T;B^{1/p}_{p,\infty}\cap L^{\infty})$ if $<1+\frac{1}{p}$. Then the transport equation (T) has a unique solution $ f\in L^{\infty}([0,T];B^s_{p,r})\cap (\bigcap _{s'<s}C([0,T];B^{s'}_{p,1})$ and the inequalities in Lemma \ref{lem2.3} hold true. Moreover, if $r<\infty$, then we have $f\in C([0,T];B^s_{p,r})$.
\end{lemma}
\begin{lemma}\label{lem2.5}(\cite{RD2}) Denote $\overline{\mathbb{N}}=\mathbb{N}\cup\infty$. Let $(v^{(n)})_{n\in \overline{\mathbb{N}}}$ be a sequence of functions belonging to $C([0,T];B^{1/2}_{2,1})$. Assume that $v^{(n)}$ is the solution to
\begin{equation}\label{2.4}
\begin{cases}
\partial_tv^{(n)}+a^{(n)}\partial_xv^{(n)}=f,\\
v^{(n)}|_{t=0}=v_0
\end{cases}
\end{equation}
with $v_0 \in B^{1/2}_{2,1}$, $f\in L^1(0,T;B^{1/2}_{2,1})$ and that for some $\alpha  \in L^1(0,T)$,
\[
\sup_{n\in\mathbb{N}}\|\partial_xa^{(n)}(t)\|_{B^{1/2}_{2,1}}\leq \alpha(t).
\]
If in addition $a^{(n)}$ tends to $a^{(\infty)}$ in $L^1(0,T;B^{1/2}_{2,1})$ then $v^{(n)}$ tends to $v^{(\infty)}$ in $C(0,T;B^{1/2}_{2,1})$.
\end{lemma}
\section{Commutator estimates}
This section is devoted to the vital commutator estimates needed in the proof of Theorem \ref{the1.1} and Theorem \ref{the1.2}. To begin with, we give the definition of the Bony decomposition.
\begin{definition}\cite{RD4}\label{def} The nonhomogeneous paraproduct of $v$ by $u$ is defined by
\begin{equation}\label{def1}
T_uv\overset{def}{=}\sum_jS_{j-1}u\Delta_jv,
\end{equation}
and the nonhomogeneous remainder of $u$ and $v$ is defined by
\begin{equation}\label{def2}
R(u,v)=\sum_{|k-j|\leq 1}\Delta_ku\Delta_jv.
\end{equation}
Then the Bony decomposition is given by
\begin{equation}\label{def3}
uv=T_uv+T_vu+R(u,v).
\end{equation}
or
\begin{equation}\label{def4}
uv=T_uv+T'_vu\quad \text{with} \quad
T'_vu\overset{def}{=}\sum_jS_{j+2}v\Delta_ju.
\end{equation}
\end{definition}
The required properties which will be used in the proof the commutator estimates are given by the following lemmas.
\begin{lemma}\cite{RD4}\label{com}
Let $f$ be a smooth function on $\mr$. Assume that $f$ is homogeneous of degree $m$ away from a neighborhood of 0.
Let $\rho$ be in $(0, 1)$, $s$ be in $\mr$, and $(p, r)$ be in $[1,\infty]^2$. There exists a constant $C$, depending only on $s$ and $\rho$, such that if $(p_1,p_2)\in [1,\infty]^2$ satisfies $\frac 1 p=\frac 1 {p_1}+\frac 1 {p_2}$, then the following estimate holds true:
\begin{equation}\label{c1}
\|[T_a,f(D)]u\|_{B^{s-m+\rho}_{p,r}}
\leq C\|\p_x a\|_{B^{\rho-1}_{p_1,\infty}}\|u\|_{B^{s}_{p_2,r}}.
\end{equation}
In the limit case $\rho = 1$, we have
\begin{equation}\label{c2}
\|[T_a,f(D)]u\|_{B^{s-m+1}_{p,r}}
\leq C\|\p_x a\|_{L^{p_1}}\|u\|_{B^{s}_{p_2,r}}.
\end{equation}
\end{lemma}
\begin{lemma}\cite{RD4}\label{est-T}
There exists a constant $C>0$, such that for any $s\in\mr,\sigma>0$ and $1\leq p,r_1,r_2\leq \infty$, it holds that
\begin{equation}\label{est-T1}
\|T_uv\|_{B^{s}_{p,r}}
\leq C^{|s|+1}\|u\|_{L^{\infty}}\|v\|_{B^{s}_{p,r}},
\end{equation}
and
\begin{equation}\label{est-T2}
\|T_uv\|_{B^{s-\sigma}_{p,r_{1,2}}}
\leq \frac {C^{|s-\sigma|+1}}{\sigma}\|u\|_{B^{-\sigma}_{\infty,r_{1}}}\|v\|_{B^{s}_{p,r_{2}}},\quad \text{with}\;\;
\frac {1}{r_{1,2}}
=\min \left(1, \frac 1 {r_{1}}+\frac 1 {r_{2}}\right).
\end{equation}
\end{lemma}
\begin{lemma}\cite{RD3,MWZ}\label{est-R}
Let $s_1,s_2\in\mr$ and $1\leq p,p_1,p_2,r_1,r_2\leq \infty$. Assume that
\begin{equation}\label{est-R1}
\frac 1 p \leq \frac 1 {p_1}+\frac 1 {p_2}\leq 1,
\quad
\sigma_{1,2}-\frac {d} {p}=\left(s_1-\frac d {p_1}\right)+\left(s_2-\frac d {p_2}\right).
\end{equation}
Then there exists a constant $C>0$ such that

(1) If
\begin{equation}\label{est-R2}
s_1+s_2>0,\quad \frac 1 {r_1}+\frac 1 {r_2}\overset{def}{=}\frac 1 r\leq 1,
\end{equation}
then we have
\begin{equation}\label{est-R3}
\|R(u,v)\|_{B^{\sigma_{1,2}}_{p,r}}
\leq \frac{C^{|s_1+s_2|+1}}{s_1+s_2} \|u\|_{B^{s_1}_{p_1,r_1}}\|v\|_{B^{s_2}_{p_2,r_2}}.
\end{equation}

(2) If
\begin{equation}\label{est-R4}
s_1+s_2=0,\quad \frac 1 {r_1}+\frac 1 {r_2}=1,
\end{equation}
then we have
\begin{equation}\label{est-R5}
\|R(u,v)\|_{B^{\sigma_{1,2}}_{p,\infty}}
\leq C^{|s_1+s_2|+1}\|u\|_{B^{s_1}_{p_1,r_1}}\|v\|_{B^{s_2}_{p_2,r_2}}.
\end{equation}
\end{lemma}
Now we are in the position to give the desired commutator estimates.

\begin{lemma}\label{com-e}
There exists a constant $C>0$ such that
\begin{equation}\label{com-e1}
\|[f,(-\p_x^2)^\n]g\|_{B^{s_0-2\nu}_{2,1}}\leq C \|f\|_{B^{s_0}_{2,1}}\|g\|_{B^{s_0-1}_{2,1}} \quad \text{for} \;\; f\in B^{s_0}_{2,1},\; g\in B^{s_0-1}_{2,1},
\end{equation}
\begin{equation}\label{com-e2}
\|[f,(-\p_x^2)^\n]g\|_{B^{s_0-1-2\nu}_{2,\infty}}\leq C \|f\|_{B^{s_0-1}_{2,1}}\|g\|_{B^{s_0-1}_{2,1}} \quad \text{for} \;\; f\in B^{s_0-1}_{2,1},\; g\in B^{s_0-1}_{2,1},
\end{equation}
and
\begin{equation}\label{com-e3}
\|[f,(-\p_x^2)^\n]g\|_{B^{s_0-1-2\nu}_{2,\infty}}\leq C \|f\|_{B^{s_0}_{2,1}}\|g\|_{B^{s_0-2}_{2,1}} \quad \text{for} \;\; f\in B^{s_0}_{2,1},\; g\in B^{s_0-2}_{2,1},
\end{equation}
where $s_0=2\nu-\frac 1 2$ for $\nu>\frac 3 2 $ and $s_0=\frac 5 2$ for $1<\nu\leq \frac 3 2 $.
\end{lemma}
\begin{proof}We write $[f,(-\p_x^2)^\n]g=F+G$ with
\begin{equation}\label{p1}
F\overset{def}{=}[T_f,(-\p_x^2)^\n]g+T_{(-\p_x^2)^\n g}f-(-\p_x^2)^\n T'_gf,
\end{equation}
and
\begin{equation}\label{p2}
G\overset{def}{=}R(f,(-\p_x^2)^\n g).
\end{equation}
First, we give the proof of \eqref{com-e1}. Taking $s=s_0-1, m=2\nu, p=2,r=1, p_1=\infty, p_2=2$ in \eqref{c2}, we have
\begin{align}\label{p3}
\|[T_f,(-\p_x^2)^\n]g\|_{B^{s_0-2\nu}_{2,1}}
&=\|[T_f,(-\p_x^2)^\n]g\|_{B^{s_0-1-2\nu+1}_{2,1}}\nonumber\\
&\leq C\|f_x\|_{L^{\infty}}\|g\|_{B^{s_0-1}_{2,1}}
\leq C \|f\|_{B^{s_0}_{2,1}}\|g\|_{B^{s_0-1}_{2,1}},
\end{align}
where the algebraic properties (4) in Lemma \ref{lem2.2} have been used. Taking $s=s_0, \sigma=2\nu, p=2, r_1=1, r_2=\infty$ in \eqref{est-T2}, we obtain
\begin{align}\label{p4}
\|T_{(-\p_x^2)^\n g}f\|_{B^{s_0-2\nu}_{2,1}}
&\leq C\|(-\p_x^2)^\n g\|_{B^{-2\nu}_{\infty,\infty}}\|f\|_{B^{s_0}_{2,1}}\nonumber\\
&\leq C\|g\|_{L^{\infty}}\|f\|_{B^{s_0}_{2,1}}
\leq C \|f\|_{B^{s_0}_{2,1}}\|g\|_{B^{s_0-1}_{2,1}},
\end{align}
where the property $L^\infty \hookrightarrow B^{0}_{\infty,\infty}$ has been used. By \eqref{est-T1} and \eqref{est-R3}, we get
\begin{align}\label{p5}
\|(-\p_x^2)^\n T'_gf\|_{B^{s_0-2\nu}_{2,1}}
\leq C\|T'_gf\|_{B^{s_0}_{2,1}}&\leq C\|T_gf\|_{B^{s_0}_{2,1}}+C\|R(f,g)\|_{B^{s_0}_{2,1}}\nonumber\\
&\leq C\|g\|_{L^{\infty}}\|f\|_{B^{s_0}_{2,1}}
\leq C \|f\|_{B^{s_0}_{2,1}}\|g\|_{B^{s_0-1}_{2,1}}.
\end{align}
Taking $s_1=s_0, s_2=-\frac 3 2, p=2, p_1=2, p_2=2, r=2, r_1=2, r_2=2, \sigma_{1,2}=s_0-2$ in \eqref{est-R3}, we obtain
\begin{align}\label{p6}
\|R(f,(-\p_x^2)^\n g)\|_{B^{s_0-2\nu}_{2,1}}
&\leq C\|R(f,(-\p_x^2)^\n g)\|_{B^{s_0-2}_{2,1}}\nonumber\\
&\leq C\|f\|_{B^{s_0}_{2,2}}\|(-\p_x^2)^\n g\|_{B^{-\frac 3 2}_{2,2}}
\leq C \|f\|_{B^{s_0}_{2,1}}\|g\|_{B^{s_0-1}_{2,1}}.
\end{align}
Combing \eqref{p1}-\eqref{p6} leads to \eqref{com-e1}. Now, we turn to the proof of \eqref{com-e2} and \eqref{com-e3}. Taking $s=s_0-2, m=2\nu, p=2,r=\infty, p_1=\infty, p_2=2$ in \eqref{c2}, we have
\begin{align}\label{p7}
\|[T_f,(-\p_x^2)^\n]g\|_{B^{s_0-1-2\nu}_{2,\infty}}
&=\|[T_f,(-\p_x^2)^\n]g\|_{B^{s_0-2-2\nu+1}_{2,\infty}}\nonumber\\
&\leq C\|f_x\|_{L^{\infty}}\|g\|_{B^{s_0-2}_{2,\infty}}
\leq C \|f\|_{B^{s_0-1}_{2,1}}\|g\|_{B^{s_0-2}_{2,1}}.
\end{align}
Taking $s=s_0-1, \sigma=2\nu, p=2, r_1=\infty, r_2=\infty$ in \eqref{est-T2}, we obtain
\begin{align}\label{p8}
\|T_{(-\p_x^2)^\n g}f\|_{B^{s_0-1-2\nu}_{2,\infty}}
&\leq C\|(-\p_x^2)^\n g\|_{B^{-2\nu}_{\infty,\infty}}\|f\|_{B^{s_0-1}_{2,\infty}}\nonumber\\
&\leq C\|g\|_{L^{\infty}}\|f\|_{B^{s_0-1}_{2,\infty}}
\leq C \|f\|_{B^{s_0-1}_{2,1}}\|g\|_{B^{s_0-2}_{2,1}}.
\end{align}
By \eqref{est-T1} and \eqref{est-R3}, we get
\begin{align}\label{p9}
\|(-\p_x^2)^\n T'_gf\|_{B^{s_0-1-2\nu}_{2,\infty}}
&\leq C\|T'_gf\|_{B^{s_0-1}_{2,\infty}}\nonumber\\
&\leq C\|g\|_{L^{\infty}}\|f\|_{B^{s_0-1}_{2,\infty}}
\leq C \|f\|_{B^{s_0-1}_{2,1}}\|g\|_{B^{s_0-2}_{2,1}}.
\end{align}
This completes the estimate of the term $F$, and we now give the estimate of the term $G$.

(1) Assume $\nu>\frac 3 2$ and thus $s_0=2\nu-\frac 1 2$.

(i) Taking $s_1=s_0-1, s_2=-\frac 3 2, p=2, p_1=2, p_2=2, r=\infty, r_1=\infty, r_2=\infty, \sigma_{1,2}=s_0-3$ in \eqref{est-R3}, we obtain
\begin{align}\label{p10}
\|R(f,(-\p_x^2)^\n g)\|_{B^{s_0-1-2\nu}_{2,\infty}}
&\leq C\|R(f,(-\p_x^2)^\n g)\|_{B^{s_0-3}_{2,\infty}}\nonumber\\
&\leq C\|f\|_{B^{s_0-1}_{2,\infty}}\|(-\p_x^2)^\n g\|_{B^{-\frac 3 2}_{2,\infty}}
\leq C \|f\|_{B^{s_0-1}_{2,1}}\|g\|_{B^{s_0-1}_{2,1}}.
\end{align}

(ii) Taking $s_1=s_0, s_2=-\frac 5 2, p=2, p_1=2, p_2=2, r_1=\infty, r_2=\infty, \sigma_{1,2}=s_0-3$ in \eqref{est-R3}, we obtain
\begin{align}\label{p11}
\|R(f,(-\p_x^2)^\n g)\|_{B^{s_0-1-2\nu}_{2,\infty}}
&\leq C\|R(f,(-\p_x^2)^\n g)\|_{B^{s_0-3}_{2,\infty}}\nonumber\\
&\leq C\|f\|_{B^{s_0}_{2,\infty}}\|(-\p_x^2)^\n g\|_{B^{-\frac 5 2}_{2,\infty}}
\leq C \|f\|_{B^{s_0}_{2,1}}\|g\|_{B^{s_0-2}_{2,1}}.
\end{align}

(2) Assume $1<\nu\leq\frac 3 2$ and thus $s_0=\frac 5 2$.

(i) Taking $s_1=\frac 3 2, s_2=-\frac 3 2, p=2, p_1=2, p_2=2, r_1=1, r_2=\infty, \sigma_{1,2}=-\frac 1 2$ in \eqref{est-R4}, we obtain
\begin{align}\label{p12}
\|R(f,(-\p_x^2)^\n g)\|_{B^{\frac 3 2-2\nu}_{2,\infty}}
&\leq C\|R(f,(-\p_x^2)^\n g)\|_{B^{-\frac 1 2}_{2,\infty}}\nonumber\\
&\leq C\|f\|_{B^{\frac 3 2}_{2,1}}\|(-\p_x^2)^\n g\|_{B^{-\frac 3 2}_{2,\infty}}
\leq C \|f\|_{B^{s_0-1}_{2,1}}\|g\|_{B^{s_0-1}_{2,1}}.
\end{align}

(ii) Taking $s_1=\frac 5 2, s_2=-\frac 5 2, p=2, p_1=2, p_2=2, r_1=1, r_2=\infty, \sigma_{1,2}=-\frac 1 2$ in \eqref{est-R4}, we obtain
\begin{align}\label{p13}
\|R(f,(-\p_x^2)^\n g)\|_{B^{\frac 3 2-2\nu}_{2,\infty}}
&\leq C\|R(f,(-\p_x^2)^\n g)\|_{B^{-\frac 1 2}_{2,\infty}}\nonumber\\
&\leq C\|f\|_{B^{\frac 5 2}_{2,1}}\|(-\p_x^2)^\n g\|_{B^{-\frac 5 2}_{2,\infty}}
\leq C \|f\|_{B^{s_0}_{2,1}}\|g\|_{B^{s_0-2}_{2,1}}.
\end{align}
Combing \eqref{p7}-\eqref{p9} with \eqref{p1}, \eqref{p2}, \eqref{p10}, \eqref{p12} yields \eqref{com-e2} and combing \eqref{p7}-\eqref{p9} with \eqref{p1}, \eqref{p2}, \eqref{p11}, \eqref{p13} yields \eqref{com-e3}. This completes the proof of Lemma \ref{com-e}.
\end{proof}

\section{Proof of Theorem \ref{the1.1}}
In this section we aim to prove Theorem \ref{the1.1} with the aid of the following six steps.

\vspace{3mm}

\emph{First step: Approximate solution}. We use a standard iterative process to build a solution. Starting from $u^{(0)}:=0$, by induction we define a sequence of smooth functions $(u^{(n)})_{n\in \mathbb{N}}$ by solving the following linear transport equation:
\begin{equation}\label{3.4}
\begin{cases}
u^{(n+1)}_t+\LC 1+u^{(n)}\RC u_x^{(n+1)}
=\p_xP(D)f_1(u^{(n)})+P(D)f_2(u^{(n)},u_x^{(n)}),\\
u^{(n+1)}(0,x)=u^{(n+1)}_0(x)=S_{n+1}u_0.\\
\end{cases}
\end{equation}
Since $S_{n+1}u_0$ belongs to $B^{\infty}_{2,r}$, by using Lemma 2.4, with the aid of induction, we show that for all $n\in \mathbb{N}$, the above equation has a global solution which belongs to $C(\mathbb{R}^{+}, B^{\infty}_{2,r})$.

\vspace{3mm}

\emph{Second step: Uniform bounds}. Applying \eqref{2.2} of Lemma \ref{lem2.3} to \eqref{3.4}, we obtain
\begin{align}\label{3.5}
&\|u^{(n+1)}(t)\|_{B^{s_0}_{2,1}}\leq e^{C\int^t_0\|u^{(n)}\|_{B^{s_0}_{2,1}}dt'}
\|u_0\|_{B^{s_0}_{2,1}}
\nonumber \\
&+\int^t_0 e^{C\int^t_\tau\|u^{(n)}\|_{B^{s_0}_{2,1}}dt'}
\|\p_xP(D)f_1(u^{(n)})+P(D)f_2(u^{(n)},u_x^{(n)})\|_{B^{s_0}_{2,1}}d\tau.
\end{align}
As $P(D)$ is a $S^{-2\n}$ -multiplier and $B^{s}_{2,r}\hookrightarrow B^{s-1}_{2,r}$, we have that
\begin{align}\label{3.6}
\|\p_xP(D)f_1(u^{(n)})\|_{B^{s_0}_{2,1}}&\leq \|f_1(u^{(n)})\|_{B^{s_0-2\n+1}_{2,1}}
\leq \|f_1(u^{(n)})\|_{B^{s_0-1}_{2,1}}\nonumber\\
&\leq \frac C 2\LC \|u^{(n)}\|_{B^{s_0}_{2,1}}+\|u^{(n)}\|^2_{B^{s_0}_{2,1}}\RC,
\end{align}
and by \eqref{com-e1} of lemma \ref{com-e} and $B^{1/2}_{2,1}\hookrightarrow L^{\infty}$, we have
\begin{align}\label{3.7}
\|P(D)f_2(u^{(n)},u_x^{(n)})\|_{B^{s_0}_{2,1}}
&\leq C\|[u^{(n)},(-\p_x^2)^\n]u^{(n)}_x\|_{B^{s_0-2\n}_{2,1}}
\leq \frac C 2\|u^{(n)}\|^2_{B^{s_0}_{2,1}}.
\end{align}
Inserting \eqref{3.6} and \eqref{3.7} into \eqref{3.5} yields for all $n\in \mathbb{N}$:
\begin{align}\label{3.8}
\|u^{(n+1)}(t)\|_{B^{s_0}_{2,1}}&\leq e^{C\int^t_0 \|u^{(n)}(t')\|_{B^{s_0}_{2,1}}dt'}
\|u_0\|_{B^{s_0}_{2,1}}\nonumber \\
&+\frac C 2 \int^t_0 e^{C\int^t_\tau  \|u^{(n)}(t')\|_{B^{s_0}_{2,1}}dt'}
\LC \|u^{(n)}\|_{B^{s_0}_{2,1}}+\|u^{(n)}\|^2_{B^{s_0}_{2,1}}\RC d\tau.
\end{align}
Let us choose a $T>0$ such that
\begin{equation}\label{3.9}
T\leq \min\{\frac{1}{C},\frac {1} {8C\|u_0\|_{B^{s_0}_{2,1}}}\},
\end{equation}
and suppose by induction that for all $t\in[0,T]$
\begin{equation}\label{3.10}
\|u^{(n)}(t)\|_{B^{s_0}_{2,1}}\leq\frac{2\|u_0\|_{{B^{s_0}_{2,1}}}}
{1-4C\|u_0\|_{{B^{s_0}_{2,1}}}t}.
\end{equation}
Indeed, one obtains from \eqref{3.8} and \eqref{3.10} that
\begin{align}\label{3.11}
\exp\{C\int^t_\tau\|(u^{(n)})(t')\|_{B^{s_0}_{2,1}}dt'\}
&\leq\exp\{\int^t_\tau\frac{2C\|u_0\|_{{B^{s_0}_{2,1}}}}
{1-4C\|u_0\|_{{B^{s_0}_{2,1}}}t'}dt'\}\nonumber \\
&=\exp\{-\frac{1}{2}\int^t_\tau
\frac{d(1-4C\|u_0\|_{{B^{s_0}_{2,1}}}t')}
{(1-4C\|u_0\|_{{B^{s_0}_{2,1}}}t')}\}\nonumber \\
&=\exp\{\frac{1}{2}\ln(\frac{1-4C\|u_0\|_{{B^{s_0}_{2,1}}}\tau}
{1-4C\|u_0\|_{{B^{s_0}_{2,1}}}t})\}\nonumber\\
&=(\frac{1-4C\|u_0\|_{{B^{s_0}_{2,1}}}\tau}
{1-4C\|u_0\|_{{B^{s_0}_{2,1}}}t})^{\frac{1}{2}}.
\end{align}
When $\tau=0$, we have
\begin{equation}\label{3.12}
e^{C\int^t_0\|(u^{(n)})(t')\|_{B^{s_0}_{2,1}} dt'}
\leq \frac{1}{(1-4C\|u_0\|_{{B^{s_0}_{2,1}}}t)^{\frac{1}{2}}}.
\end{equation}
Then combining \eqref{3.10} and \eqref{3.11}, we have
\begin{align}\label{3.13}
&\frac{C}{2}\int^t_0 e^{C\int^t_\tau\|(u^{(n)})(t')\|_{B^{s_0}_{2,1}}dt'}
(\|u^{(n)}\|_{B^{s_0}_{2,1}}+\|u^{(n)}\|^2_{B^{s_0}_{2,1}})d\tau \nonumber \\
&\leq \frac{\|u_0\|_{B^{s_0}_{2,1}}}
{({1-4C\|u_0\|_{{B^{s_0}_{2,1}}}t})^{1/2}}\nonumber \\
&\times \int^t_0 (1-4C\|u_0\|_{{B^{s_0}_{2,1}}}\tau)^{\frac{1}{2}}
 \LC\frac{C}
{(1-4C\|u_0\|_{{B^{s_0}_{2,1}}}\tau)}
+\frac{2C\|u_0\|_{B^{s_0}_{2,1}}}
{(1-4C\|u_0\|_{{B^{s_0}_{2,1}}}\tau)^{2}}\RC d\tau \nonumber\\
&=\frac{\|u_0\|_{B^{s_0}_{2,1}}}
{({1-4C\|u_0\|_{{B^{s_0}_{2,1}}}t})^{^{1/2}}}
\int^t_0\LC\frac{C}
{(1-4C\|u_0\|_{{B^{s_0}_{2,1}}}\tau)^{1/2}}
+\frac{2C\|u_0\|_{B^{s_0}_{2,1}}}
{(1-4C\|u_0\|_{{B^{s_0}_{2,1}}}\tau)^{3/2}}\RC d\tau \nonumber\\
&=\frac{\|u_0\|_{B^{s_0}_{2,1}}}
{({1-4C\|u_0\|_{{B^{s_0}_{2,1}}}t})^{^{1/2}}}
\LC\frac{Ct}
{(1-4C\|u_0\|_{{B^{s_0}_{2,1}}}\xi_2)^{1/2}}
+(1-4C\|u_0\|_{{B^{s_0}_{2,1}}}\tau)^{-1/2}\Bigg|^t_0\RC \nonumber\\
&\leq \frac{\|u_0\|_{B^{s_0}_{2,1}}}
{({1-4C\|u_0\|_{{B^{s_0}_{2,1}}}t})^{^{1/2}}}
\LC \frac{CT+1}{(1-4C\|u_0\|_{{B^{s_0}_{2,1}}}t)^{1/2}}-1 \RC,
\end{align}
where the mean-value theorem for the integral has been employed with $0<\xi_1<t$. Inserting \eqref{3.12} and \eqref{3.13} into \eqref{3.8}, we get that
\begin{align}\label{3.14}
\|u^{(n+1)}\|_{B^{s_0}_{2,1}}
\leq \frac{\LC 1+CT \RC\|u_0\|_{B^{s_0}_{2,1}}}
{{1-4C\|u_0\|_{{B^{s_0}_{2,1}}}t}}
\leq \frac{2\|u_0\|_{B^{s_0}_{2,1}}}
{{1-4C\|u_0\|_{{B^{s_0}_{2,1}}}t}}.
\end{align}
Thus, $(u^{(n)})_{n\in\mathbb{N}}$ is uniformly bounded in $C([0,T];B^{s_0}_{2,1})$. Using equation \eqref{3.4}, one can easily prove that $(\partial _tu^{(n)})_{n\in\mathbb{N}}$ is uniformly bounded in $C([0,T];B^{s_0-1}_{2,1})$. Consequently, $(u^{(n)})_{n\in\mathbb{N}} \subset C([0,T];B^{s_0}_{2,1}) \cap C^1([0,T];B^{s_0-1}_{2,1})$.

\vspace{3mm}

\emph{Third step: Convergence}. We first show that $(u^{(n)})_{n\in\mathbb{N}}$ is a Cauchy sequence in $C([0,T];B^{s_0-1}_{2,\infty})$, then by using (2.1) we prove that $(u^{(n)})_{n\in\mathbb{N}}$ is a Cauchy sequence in $C([0,T];B^{s_0-1}_{2,1})$. For $(m,n) \in \mathbb{N}^2$, we have
\begin{align}\label{3.15}
&\LB\partial_t+\LC 1+u^{(n+m)}\RC\partial_x\RB(u^{(n+1+m)}-u^{(n+1)})\nonumber\\
&=-\LC u^{(n+m)}-u^{(n)}\RC u^{(n+1)}_x+\p_xP(D)\LB\LC u^{(n+m)}-u^{(n)}\RC+\LC (u^{(n+m)})^2-(u^{(n)})^2\RC\RB\nonumber\\
&\;\quad+P(D)\LC[u^{(n+m)},(-\p_x^2)^\n]u^{(n+m)}_x-[u^{(n)},(-\p_x^2)^\n]u^{(n)}_x\RC\nonumber\\
&:=S_1(x,t)+S_2(x,t)+S_3(x,t).
\end{align}
We define
\begin{equation}\label{3.16}
w_{n,m}=\|(u^{(n+m)}-u^{(n)})(t)\|_{B^{s_0-1}_{2,\infty}}
\end{equation}
and
\begin{equation}\label{3.17}
w_{n}(t)=\sup _{m\in\mathbb{N}}w_{n,m}(t)
\end{equation}
as well as
\begin{equation}\label{3.18}
\widetilde{w}(t)=\limsup_{n\rightarrow \infty}w_{n}(t).
\end{equation}
We will show $\widetilde{w}(t)=0$, for $t \in [0,T] $. By \eqref{3.9}, \eqref{3.10} and \eqref{3.12}, we have that
\begin{equation}\label{3.19}
\|u^{(n)}(t)\|_{B^{s_0}_{2,1}}\leq 4\|u_0\|_{B^{s_0}_{2,1}}
\end{equation}
and
\begin{equation}\label{3.20}
e^{C\int^t_0\|(u^{(n)})(t')\|_{B^{s_0}_{2,1}} dt'}\leq 2.
\end{equation}
Using (2)-(5) and (10) of Lemma \ref{lem2.2} as well as \eqref{com-e2}-\eqref{com-e3} of Lemma \ref{com-e} and the above inequality, we obtain
\begin{align}
&\|(u^{(n+m)}-u^{(n)})u_x^{(n+1)}\|_{B^{s_0-1}_{2,\infty}}\nonumber\\
&\leq C\left(\|u^{(n+m)}-u^{(n)}\|_{B^{s_0-1}_{2,\infty}}
\|u_x^{(n+1)}\|_{L^{\infty}}
+\|u_x^{(n+1)}\|_{B^{s_0-1}_{2,\infty}}
\|u^{(n+m)}-u^{(n)}\|_{L^{\infty}}\right)\nonumber\\
&\leq C\|u^{(n+m)}-u^{(n)}\|_{B^{s_0-1}_{2,1}}\|u^{(n+1)}\|_{B^{s_0}_{2,1}},\nonumber
\end{align}
\begin{align}
\!\!\!\!\!\!\!\|\p_xP(D)\LC u^{(n+m)}-u^{(n)}\RC\|_{B^{s_0-1}_{2,\infty}}
\leq \|u^{(n+m)}-u^{(n)}\|_{B^{s_0-2\nu}_{2,\infty}}
\leq C\|u^{(n+m)}-u^{(n)}\|_{B^{s_0-1}_{2,1}},\nonumber
\end{align}
\begin{align}
&\|\p_xP(D)\LC (u^{(n+m)})^2-(u^{(n)})^2\RC\|_{B^{s_0-1}_{2,\infty}}\nonumber\\
&\leq C\|(u^{(n+m)})^2-(u^{(n)})^2\|_{B^{s_0-2\nu}_{2,\infty}}\nonumber\\
&\leq C\|(u^{(n+m)}-u^{(n)})(u^{(n+m)}+u^{(n)})\|_{B^{s_0-1}_{2,\infty}}\nonumber\\
&\leq C\left(\|u^{(n+m)}-u^{(n)}\|_{B^{s_0-1}_{2,\infty}}
\|u^{(n+m)}+u^{(n)}\|_{L^{\infty}}
+\|u^{(n+m)}+u^{(n)}\|_{B^{s_0-1}_{2,\infty}}
\|u^{(n+m)}-u^{(n)}\|_{L^{\infty}}\right)\nonumber\\
&\leq C\|u^{(n+m)}-u^{(n)}\|_{B^{s_0-1}_{2,1}}
\LC\|u^{(n+m)}\|_{B^{s_0}_{2,1}}+\|u^{(n)}\|_{B^{s_0}_{2,1}}\RC,\nonumber
\end{align}
\begin{align}
&\|P(D)\LC[u^{(n+m)},(-\p_x^2)^\n]u^{(n+m)}_x-[u^{(n)},(-\p_x^2)^\n]u^{(n)}_x\RC\|_{B^{s_0-1}_{2,\infty}}\nonumber\\
&=\|P(D)\LC[(u^{(n+m)}-u^{(n)}),(-\p_x^2)^\n]u^{(n+m)}_x-[u^{(n)},(-\p_x^2)^\n]\LC u^{(n+m)}_x-u^{(n)}_x\RC \RC\|_{B^{s_0-1}_{2,\infty}}\nonumber\\
&\leq \|[u^{(n+m)}-u^{(n)},(-\p_x^2)^\n]u^{(n+m)}_x\|_{B^{s_0-1-2\n}_{2,\infty}}
+\|[u^{(n)},(-\p_x^2)^\n]\LC u^{(n+m)}_x-u^{(n)}_x\RC\|_{B^{s_0-1-2\n}_{2,\infty}}\nonumber\\
&\leq C\|u^{(n+m)}-u^{(n)}\|_{B^{s_0-1}_{2,1}}\|u^{(n+m)}_x\|_{B^{s_0-1}_{2,1}}
+C\|u^{(n)}\|_{B^{s_0}_{2,1}}\|u_x^{(n+m)}-u_x^{(n)}\|_{B^{s_0-2}_{2,1}}\nonumber\\
&\leq C\|u^{(n+m)}-u^{(n)}\|_{B^{s_0-1}_{2,1}}
\LC\|u^{(n+m)}\|_{B^{s_0}_{2,1}}+\|u^{(n)}\|_{B^{s_0}_{2,1}}\RC.\nonumber
\end{align}
We define
\begin{equation}\label{3.21}
M=24\|u_0\|_{B^{s_0}_{2,1}}+2,
\end{equation}
then we have from above inequalities and \eqref{3.19} that
\begin{align}\label{3.22}
&\|S_1(x,t)+S_2(x,t)+S_3(x,t)\|_{B^{s_0-1}_{2,\infty}}\nonumber\\
&\leq C\|u^{(n+m)}-u^{(n)}\|_{B^{s_0-1}_{2,1}}\LC 1+\|u^{(n+m)}\|_{B^{s_0}_{2,1}}+\|u^{(n)}\|_{B^{s_0}_{2,1}}
+\|u^{(n+1)}\|_{B^{s_0}_{2,1}}\RC\nonumber\\
&\leq C \frac {M} 2 \|u^{(n+m)}-u^{(n)}\|_{B^{s_0-1}_{2,1}}.
\end{align}
Note that
\begin{align}\label{3.23}
\|u^{(n+1+m)}_0-u^{(n+1)}_0\|_{B^{s_0-1}_{2,\infty}}
&=\|(S_{n+1+m}u_0-S_{n+1}u_0)\|_{B^{s_0-1}_{2,\infty}}=\|\sum ^{n+m}_{q=n+1}\triangle _qu_0\|_{B^{s_0-1}_{2,\infty}}\nonumber\\
&=\sup_{k\geq1}2^{(s_0-1)k}\|\triangle _k(\sum ^{n+m}_{q=n+1}\triangle _qu_0)\|_{L^2}\nonumber\\
&=\sup_{n+1\leq k\leq n+m+1}2^{-k}2^{s_0k}
\|\triangle_{k-1}\triangle _ku_0+\triangle_{k+1}\triangle _ku_0\|_{L^2}\nonumber\\
&\leq \sup_{n\leq k\leq n+m}2^{-k}2^{s_0k}\|\triangle _ku_0\|_{L^2}
\leq C 2^{-n}\|u_0\|_{B^{s_0}_{2,1}}.
\end{align}
Applying \eqref{2.2} of Lemma \ref{lem2.3} and using \eqref{3.20}-\eqref{3.23}, we have for $t\in[0,T]$,
\begin{align}\label{3.24}
&\|(u^{(n+1+m)}-u^{(n+1)})(t)\|_{B^{s_0-1}_{2,\infty}}\nonumber\\
&\leq e^{C\int^t_0\|u^{(n+m)}(t')\|_{B^{s_0-1}_{2,\infty}}dt'} \|u^{(n+1+m)}_0-u^{(n+1)}_0\|_{B^{s_0-1}_{2,\infty}}\nonumber\\
&\qquad+\int^t_0 e^{C\int^t_\tau\|u^{(n+m)}(t')\|_{B^{s_0-1}_{2,\infty}}dt'}
\|S_1(x,t)+S_2(x,t)+S_3(x,t)\|_{B^{s_0-1}_{2,\infty}}d\tau\nonumber\\
&\leq CM2^{-n}
+CM\int^t_0\|u^{(n+m)}-u^{(n)}\|_{B^{s_0-1}_{2,1}}d\tau.
\end{align}
Combing \eqref{3.24}, \eqref{3.16} and (11) of Lemma \ref{lem2.2}, we know that for $\forall(n,m)\in \mathbb{N}^2$
\begin{align}\label{3.25}
w_{n+1,m}&=\|u^{(n+1+m)}-u^{(n+1)}\|_{B^{s_0-1}_{2,\infty}}\nonumber\\
&\leq CM\LC 2^{-n}+\int^t_0\|(u^{(n+m)}-u^{(n)})\|_{B^{s_0-1}_{2,1}}d\tau\RC\nonumber\\
&\leq CM\LB2^{-n}+\int^t_0\|(u^{(n+m)}-u^{(n)})\|_{B^{s_0-1}_{2,\infty}}
\ln \LC e+\frac{\|(u^{(n+m)}-u^{(n)})\|_{B^{s_0}_{2,\infty}}}
{\|(u^{(n+m)}-u^{(n)})\|_{B^{s_0-1}_{2,\infty}}}\RC d\tau\RB\nonumber\\
&\leq CM\LB2^{-n}+\int^t_0w_{n,m}(\tau)\ln \LC e+\frac{M}{w_{n,m}(\tau)}\RC d\tau\RB.
\end{align}
By \eqref{3.17} and \eqref{3.25}, we have
\begin{equation}\label{3.26}
w_{n+1}\leq CM\LB2^{-n}+\int^t_0w_{n}(\tau)\ln \LC e+\frac{M}{w_{n}(\tau)}\RC d\tau\RB.
\end{equation}
Letting $n\rightarrow+\infty$ in \eqref{3.26} yields
\begin{equation}\label{3.27}
\widetilde{w}(t)\leq CM\int^t_0\widetilde{w}(\tau)\ln (e+\frac{M}{\widetilde{w}(\tau)})d\tau.
\end{equation}
Because for $x\in(0,1]$ and $\alpha>0$, we have
\begin{equation}\label{3.28}
\ln(e+\frac{\alpha}{x})\leq \ln(e+\alpha)(1-\ln x).
\end{equation}
Then the inequality \eqref{3.27} can be rewritten as
\begin{equation}\label{3.29}
\widetilde{w}(t)\leq CM\int^t_0\widetilde{w}(\tau)\ln (e+M)(1-\ln \widetilde{w}(\tau))d\tau
\end{equation}
provided that $\widetilde{w}(t)\leq1$ on $[0,T]$. Using a Gronwall type argument
(see e.g. Lemma 5.2.1 in \cite{Ch}) yields $\widetilde{w}(t)=0$ for $t\in[0,T]$.

Now we claim that $(u^{(n)})_{n\in \mathbb{N}}$ is a Cauchy sequence in $C([0,T];B^{s_0-1}_{2,1})$. Using \eqref{2.1} of Lemma \ref{lem2.2}, we have that
\begin{align}\label{3.30}
&\|(u^{(n+1+m)}-u^{(n+1)})(t)\|_{B^{s_0-1}_{2,1}}\nonumber\\
&\leq C(\theta)\|(u^{(n+1+m)}-u^{(n+1)})(t)\|_{B^{s_0-1}_{2,\infty}}^{\theta}
\|(u^{(n+1+m)}-u^{(n+1)})(t)\|_{B^{s_0}_{2,\infty}}^{1-\theta}\nonumber\\
&\leq C(\theta)\|(u^{(n+1+m)}-u^{(n+1)})(t)\|_{B^{s_0-1}_{2,\infty}}^{\theta}
\LC
\|u^{(n+1+m)}(t)\|_{B^{s_0}_{2,1}}+\|u^{(n+1)}(t)\|_{B^{s_0}_{2,1}}\RC^{1-\theta}\nonumber\\
&\leq C(\theta)(CM)^{1-\theta}\|(u^{(n+1+m)}-u^{(n+1)})(t)\|_{B^{s_0-1}_{2,\infty}}^{\theta}\nonumber\\
&= C(\theta)(CM)^{1-\theta}w^{\theta}_{n+1,m}(t).
\end{align}
For $\forall t\in[0,T]$, $m\in \mathbb{N}$, we get from \eqref{3.30} that
\[
\limsup_{n\rightarrow \infty}\|(u^{(n+1+m)}-u^{(n+1)})(t)\|_{B^{s_0-1}_{2,1}}=0.
\]
Thus, $(u^{(n)})_{n\in \mathbb{N}}$ is a Cauchy sequence in $C([0,T];B^{s_0-1}_{2,1})$, whence $(u^{(n)})_{n\in \mathbb{N}}$ converges to some limit $u \in C([0,T];B^{s_0-1}_{2,1})$.

\vspace{3mm}

\emph{Fourth step: Existence and continuity of solution in $E^{s_0}_{2,1}(T)$}. Now we have to check that $u$ belongs to $E^{s_0}_{2,1}(T)$ and satisfies \eqref{3.2}. Since $(u^{(n)})_{n\in \mathbb{N}}$ is uniformly bounded in $L^{\infty}([0,T];B^{s_0}_{2,1})$. From (8) of Lemma \ref{lem2.2}, we have that $u\in L^{\infty}([0,T];B^{s_0}_{2,1})$. From \eqref{3.2}, we can easily prove that $u_t\in L^{\infty}([0,T];B^{s_0-1}_{2,1})$. It is easily checked that $u$ is indeed a solution to \eqref{3.2} by passing to the limit in \eqref{3.4}. Using similar proof to \cite{RD1}, we can obtain that $u\in E^{s_0}_{2,1}(T)$.

\vspace{3mm}

\emph{Fifth step: Uniqueness}. Uniqueness is a corollary of the following result.
\begin{proposition}\label{pro3.1}
Let $v$, $u$ be solutions to the problem \eqref{3.2} with initial data $v_0$, $u_0$, respectively.  Let $w(t):=v-u$. Obviously, $w(0):=v_0-u_0$. There exists a constant $C$ such that if for some $T_\star\leq T$
\[
\sup_{t\in[0,T_\star]}\LC e^{-C\int^t_0
\|u(\tau)\|_{B^{s_0-1}_{2,\infty}} d\tau}\|w(t)\|_{B^{s_0-1}_{2,\infty}}\RC \leq1,
\]
then the following inequality holds true for $t\in[0,T_\star]$:
\begin{equation}\label{3.31}
\|w(t)\|_{B^{s_0-1}_{2,\infty}}\leq e^{\LC 1+C\int^t_0\|u(\tau)\|_{B^{s_0-1}_{2,\infty}} d\tau\RC}
\LC \frac{\|w(0)\|_{B^{s_0-1}_{2,\infty}}}{e}\RC ^{\exp[-CtZ\ln(e+Z)]},
\end{equation}
where $Z$ is defined as
\[
Z=4\|u_0\|_{B^{s_0}_{2,1}}+4\|v_0\|_{B^{s_0}_{2,1}}+1.
\]
Furthermore , if
\begin{equation}\label{3.32}
\|w(0)\|_{B^{s_0-1}_{2,\infty}}\leq e^{1-\exp[CTZ\ln(e+Z)]},
\end{equation}
then \eqref{3.31} is valid on $[0,T]$. In particular, when $w(0)=0$, then $u(x,t)=v(x,t)$.
\end{proposition}
\begin{proof}
Obviously, $w$ solves the following Cauchy problem for the transport equation:
\begin{align}\label{3.33}
& w_t+\LC 1+u\RC w_x\nonumber\\
&=-wv_x
+\p_xP(D)\LC w+w(v+u)\RC
+P(D)\LC [w,(-\p_x^2)^\n]v_x-[u,(-\p_x^2)^\n]w_x\RC\nonumber\\
&:=\widetilde{S_1}(x,t)+\widetilde{S_2}(x,t)+\widetilde{S_3}(x,t).
\end{align}
Using \eqref{2.2} of Lemma \ref{lem2.3} and \eqref{3.33}, we have
\begin{align}\label{3.34}
&\|w(t)\|_{B^{s_0-1}_{2,\infty}}\leq \|w(0)\|_{B^{s_0-1}_{2,\infty}}
e^{C\int^t_0\|u(\tau)\|_{B^{s_0-1}_{2,\infty}} d\tau}\nonumber\\
&+\int^{t}_{0}e^{C\int^t_{\tau}\|u(t')\|_{B^{s_0-1}_{2,\infty}} dt'}\|\widetilde{S_1}(x,t)+\widetilde{S_2}(x,t)+\widetilde{S_3}(x,t)\|_{B^{s_0-1}_{2,\infty}},
d\tau.
\end{align}
Following the proof of \eqref{3.22}, we obtain
\begin{equation}\label{3.35}
\|\widetilde{S_1}(x,t)+\widetilde{S_2}(x,t)+\widetilde{S_3}(x,t)\|_{B^{s_0-1}_{2,\infty}}
\leq C Z \|w\|_{B^{s_0-1}_{2,1}}.
\end{equation}
Inserting \eqref{3.35} into \eqref{3.34} yields
\begin{align}\label{3.36}
&\|w(t)\|_{B^{s_0-1}_{2,\infty}}\nonumber\\
&\leq \|w(0)\|_{B^{s_0-1}_{2,\infty}}
e^{C\int^t_0\|u(\tau)\|_{B^{s_0-1}_{2,\infty}} d\tau}
+CZ\int^{t}_{0}e^{C\int^t_{\tau}\|u(t')\|_{B^{s_0-1}_{2,\infty}} dt'} \|w\|_{B^{s_0-1}_{2,1}}d\tau\nonumber\\
&\leq\|w(0)\|_{B^{s_0-1}_{2,\infty}}
e^{C\int^t_0\|u(\tau)\|_{B^{s_0-1}_{2,\infty}} d\tau}
+CZ\int^{t}_{0}e^{C\int^t_{\tau}\|u(t')\|_{B^{s_0-1}_{2,\infty}} dt'} \|w\|_{B^{s_0-1}_{2,\infty}}\ln(e+\frac{\|w\|_{B^{s_0}_{2,\infty}}}
{\|w\|_{B^{s_0-1}_{2,\infty}}})d\tau\nonumber\\
&\leq CZ\int^{t}_{0}e^{C\int^t_{\tau}\|u(t')\|_{B^{s_0-1}_{2,\infty}} dt'} \|w\|_{B^{s_0-1}_{2,\infty}}\ln\LC e+\frac{Z}
{e^{-C\int^{\tau}_0\|u(t')\|_{B^{s_0-1}_{2,\infty}} dt'}\|w\|_{B^{s_0-1}_{2,\infty}}}\RC d\tau\nonumber\\
&\quad+
\|w(0)\|_{B^{s_0-1}_{2,\infty}}
e^{C\int^t_0\|u(\tau)\|_{B^{s_0-1}_{2,\infty}} d\tau}.
\end{align}
Denote
\[
W(t)=e^{-C\int^{t}_0\|u(\tau)\|_{B^{s_0-1}_{2,\infty}} d\tau}\|w\|_{B^{s_0-1}_{2,\infty}}.
\]
Inequality \eqref{3.36} can be rewritten as
\begin{align}
W(t)\leq W(0)+CZ\int^t_0W(\tau)\ln(e+\frac{Z}{W(\tau)})d\tau
\end{align}
In light of the hypothesis and using a Gronwall type argument \cite{YYY} yields
\[
\frac{W(t)}{e}\leq(\frac{W(0)}{e})^{\exp[-CtZ\ln(e+Z)]},
\]
implying the desired result. \eqref{3.32} implies that \eqref{3.31} is valid with $T_\star=T$.
\end{proof}

\vspace{3mm}

\emph{Sixth step: Continuity with respect to the initial data in $B^{s_0}_{2,1}$}.
\begin{proposition} \label{pro3.2}
For any $u_0\in B^{s_0}_{2,1}$, there exist a $T>0$ and a neighborhood $V$ of $u_0$ in $B^{s_0}_{2,1}$ such that the map
\[\Phi:
\begin{cases}
V\subset B^{s_0}_{2,1}\rightarrow C([0,T];B^{s_0}_{2,1}),\\
v_0\rightarrow v \;\text{solution to \eqref{3.2} with initial datum $v_0$}\\
\end{cases}
\]
is continuous.
\end{proposition}
\begin{proof}Motivated by \cite{RD2}, we prove Proposition \ref{pro3.2} by using Lemma \ref{lem2.5}.

\emph{First step: Continuity in $C([0,T];B^{s_0-1}_{2,1})$}. For $u_0\in B^{s_0}_{2,1}$ and $r>0$, we claim that there exist a $T>0$ and a $M>0$ such that for any $u'_0\in B^{s_0}_{2,1}$ with $\|u_0-u'_0\|_{B^{s_0}_{2,1}}\leq r$, the solution $u'=\Phi(u'_0)$ of \eqref{3.2} associated with $u'_0$ belongs to $C([0,T];B^{s_0}_{2,1})$ and satisfies
\[
\|u'\|_{L^{\infty}(0,T;B^{s_0}_{2,1})}\leq M.
\]
Indeed, from
\[
\|u'\|_{B^{s_0}_{2,1}}\leq\frac{2\|u'_0\|_{{B^{s_0}_{2,1}}}}
{1-4C\|u'_0\|_{{B^{s_0}_{2,1}}}t},
\]
we know that $T<\frac{1}{4C\|u'_0\|_{{B^{s_0}_{2,1}}}}$. Thus we can choose
\[
T=\frac{1}{8C\LC(\|u_0\|_{{B^{s_0}_{2,1}}}+r)+r\RC},\quad M=4\|u_0\|_{{B^{s_0}_{2,1}}}+4r.
\]
Then
\[
T\leq\frac{1}{8C(\|u'_0\|_{{B^{s_0}_{2,1}}}+r)},
\]
and
\[
\|u'\|_{B^{s_0}_{2,1}}\leq \frac{2\|u'_0\|_{{B^{s_0}_{2,1}}}}{1-\frac{\|u'_0\|_{{B^{s_0}_{2,1}}}}
{2(\|u'_0\|_{{B^{s_0}_{2,1}}}+r)}}
\leq4\|u'_0\|_{{B^{s_0}_{2,1}}}\leq M.
\]
Combining the above uniform bounds with Proposition \ref{pro3.1}, we infer that
\[
\|\Phi(u'_0)-\Phi(u_0)\|_{L^{\infty}(0,T;{B^{s_0-1}_{2,\infty}})}
\leq e^{(1+CMT)}(\frac{\|u'_0-u_0\|_{B^{s_0-1}_{2,\infty}}}{e})^{\exp[-CTZ\ln(e+Z)]}
\]
provided that
\[
\|u'_0-u_0\|_{B^{s_0-1}_{2,\infty}}\leq e^{1-\exp[CTZ\ln(e+Z)]}.
\]
In view of the uniform bounds in $C([0,T];B^{s_0}_{2,1})$ and an interpolation argument, we infer the map $\Phi$ is continuous from $B^{s_0}_{2,1}$ into $C([0,T];B^{s_0-1}_{2,1})$.

\emph{Second step: Continuity in $C([0,T];B^{s_0}_{2,1})$}. Let $u_0^{(\infty)} \in B^{s_0}_{2,1}$ and $(u_0^{(n)})_{n\in \mathbb{N}}$ tend to $u_0^{(\infty)}$ in $B^{s_0}_{2,1}$. We denote by $u^{(n)}$ the solution with the initial data $u_0^{(n)}$. From the first step, we can find $T$, $M>0$ such that for all $n\in \mathbb{N}$, $u^{(n)}$ is defined on $[0,T]$ and
\[
\sup_{n\in \overline{\mathbb{N}}}\|u^{(n)}\|_{L^{\infty}(0,T;{B^{s_0}_{2,1}})}\leq M.
\]
Thanks to step one, proving that $u^{(n)}$ tends to $u^{(\infty)}$ in $C([0,T];B^{s_0}_{2,1})$ amounts to proving that $v^{(n)}=\partial_xu^{(n)}$ tends to $v^{(\infty)}=\partial_xu^{(\infty)}$ in $C([0,T];B^{s_0-1}_{2,1})$. Notice that $v^{(n)}$ solves the following linear transport equations
\[
\begin{cases}
\partial_tv^{(n)}+\LC1+u^{(n)}\RC\partial_xv^{(n)}=\widetilde{f}^{(n)},\\
v^{(n)}|_{t=0}=\partial_xu^{(n)}_0,
\end{cases}
\]
with
\begin{align}
\widetilde{f}^{(n)}=-(u_x^{(n)})^2+\p^2_xP(D)f_1(u^{(n)})+\p_xP(D)f_2(u^{(n)},u_x^{(n)}).\nonumber
\end{align}
Following the method in \cite{Ka1}, we decompose $v^{(n)}=v_1^{(n)}+v_2^{(n)}$ with
\[
\begin{cases}
\partial_tv_1^{(n)}+\LC1+u^{(n)}\RC\partial_xv_1^{(n)}
=\widetilde{f}^{(n)}
-\widetilde{f}^{(\infty)},\\
v_1^{(n)}|_{t=0}=\partial_xu^{(n)}_0-\partial_xu^{(\infty)}_0,
\end{cases}
\]
and
\[
\begin{cases}
\partial_tv_2^{(n)}+\LC1+u^{(n)}\RC\partial_xv_2^{(n)}
=\widetilde{f}^{(\infty)},\\
v_2^{(n)}|_{t=0}=\partial_xu^{(\infty)}_0.
\end{cases}
\]
On the other hand, we have
\begin{equation*}
\!\!\!\!\!\!\!\!\!\!\!\!\!\!\!\!\!\!\!\!\!\!\!\!\!\!\!\!\!\!\!\!\!\!\!\!\!\!\!\!\!\!\!\!\!\!\!\!\!\!\!\!\!\!\!\!\!\!
\!\!
\|(u_x^{(n)})^2\|_{B^{s_0-1}_{2,1}}\leq \|u_x^{(n)}\|^2_{B^{s_0-1}_{2,1}}\leq \|u^{(n)}\|^2_{B^{s_0}_{2,1}},
\end{equation*}
\begin{align*}
\qquad\quad\|\p^2_xP(D)f_1(u^{(n)})\|_{B^{s_0-1}_{2,1}}&\leq \|u^{(n)}+(u^{(n)})^2\|_{B^{s_0+1-2\nu}_{2,1}}\nonumber\\
&\leq \|u^{(n)}+(u^{(n)})^2\|_{B^{s_0-1}_{2,1}}\leq\|u^{(n)}\|_{B^{s_0}_{2,1}}+\|u^{(n)}\|^2_{B^{s_0}_{2,1}},
\end{align*}
\begin{align*}
\|\p_xP(D)f_2(u^{(n)},u_x^{(n)})\|_{B^{s_0-1}_{2,1}}&\leq \|[u^{(n)},(-\p_x^2)^\n]u_x^{(n)}\|_{B^{s_0-2\nu}_{2,1}}\nonumber\\
&\leq\|u^{(n)}\|_{B^{s_0}_{2,1}}\|u_x^{(n)}\|_{B^{s_0-1}_{2,1}}
\leq \|u^{(n)}\|^2_{B^{s_0}_{2,1}},
\end{align*}
and thus $(\widetilde{f}^{(n)})_{n\in\mathbb{N}}$ is uniformly bounded in $C([0,T];B^{s_0}_{2,1})$. A similar argument yields the following inequalities
\begin{align}
&\|\widetilde{f}^{(n)}-\widetilde{f}^{(\infty)}\|_{B^{s_0-1}_{2,1}}\nonumber\\
&\leq C\LC 1+\|u^{(n)}\|_{B^{s_0}_{2,1}}+\|u^{(\infty)}\|_{B^{s_0}_{2,1}}\RC
\LC\|u^{(n)}-u^{(\infty)}\|_{B^{s_0}_{2,1}}
+\|u_x^{(n)}-u_x^{(\infty)}\|_{B^{s_0-1}_{2,1}}\RC.\nonumber
\end{align}
Applying Lemma \ref{lem2.3}, one can deduce that
\begin{align}\label{3.37}
\|v_1^{(n)}(t)\|_{B^{s_0-1}_{2,1}}&\leq
e^{C\int^t_0\|u(\tau)\|_{B^{s_0}_{2,1}}d\tau}
\|\partial_xu^{(n)}_0-\partial_xu^{(\infty)}_0\|_{B^{s_0-1}_{2,1}}
\nonumber\\
&+C\int^{t}_{0}e^{C\int^t_{\tau}\|u_x(\tau')\|_{B^{s_0-1}_{2,1}}d\tau'} \|\widetilde{f}^{(n)}-\widetilde{f}^{(\infty)}\|_{B^{s_0-1}_{2,1}}d\tau\nonumber\\
&\leq
e^{C\int^t_0\|u(\tau)\|_{B^{s_0}_{2,1}}d\tau}
\|\partial_xu^{(n)}_0-\partial_xu^{(\infty)}_0\|_{B^{s_0-1}_{2,1}}
\nonumber\\
&+C\int^{t}_{0}e^{C\int^t_{\tau}\|u_x(\tau')\|_{B^{s_0-1}_{2,1}}d\tau'}
\LC\|u^{(n)}-u^{(\infty)}\|_{B^{s_0-1}_{2,1}}
+\|u_x^{(n)}-u_x^{(\infty)}\|_{B^{s_0-1}_{2,1}}\RC\nonumber\\
&\times\LC 1+\|u^{(n)}\|_{B^{s_0}_{2,1}}+\|u^{(\infty)}\|_{B^{s_0}_{2,1}}\RC d\tau.
\end{align}
Applying similar arguments as in \cite{RD2} on P. 441 to \eqref{3.37}, we have
\[
\partial _xu^{(n)}\rightarrow \partial _xu^{(\infty)}\quad \text{in}\;B^{s_0-1}_{2,1}.
\]
We have completed the proof of Proposition \ref{pro3.2}.
\end{proof}

Summing up the above six steps, we get Theorem \ref{the1.1}.

\section{Proof of Theorem \ref{the1.2}}
In this section, we are devoted to establishing the existence and uniqueness of analytic solutions to the system \eqref{3.2} on the line $\mr$.

The proof of Theorem \ref{the1.2} needs a suitable scale of Banach spaces as follows. For any $s>0$, we set
\begin{equation}\label{5.1}
E_s=\LCB u\in C^{\infty}(\mr): \normmm u_s=\sup_{k\in\N_0}\frac {s^k\|\p_x^ku\|_{B^{s_0}_{2,1}}}{k!/(k+1)^2}<\infty\RCB,
\end{equation}
where $\N_0$ is the set of nonnegative integers. We take note that the above space is similar to the one introduced in \cite{HM}, where $B^{s_0}_{2,1}$ is replaced by $H^2$. It is easy to verify that $E_s$ equipped with the norm $\normmm \cdot$ is a Banach space and that for any $0<s'<s$, $E_s$ is continuously embedded in $E_{s'}$ with
\begin{equation}\label{5.2}
\normmm u_{s'}\leq \normmm u_s.
\end{equation}
By this definition, one can easily get that $u$ in $E_s$ is a real analytic function on $\mr$ and what is crucial for our purposes is the fact that each $E_s$ forms an algebra under pointwise multiplication of functions.
\begin{lemma}\label{lem5.1}
(1) Let $0<s<1$. There is a constant $C>0$, independent of $s$, such that for any $u$ and $v$ in $E_s$ we have
\begin{equation}\label{5.3}
\normmm {uv}_{s}\leq C\normmm u_s\normmm v_s.
\end{equation}
(2) There is a constant $C>0$ such that for any $0<s'<s<1$, we have
\begin{equation}\label{e1}
\normmm {\p_xu}_{s'}\leq \frac {C}{s-s'}\normmm u_s,
\end{equation}
\begin{equation}\label{e2}
\normmm {P(D)u}_{s}\leq C\normmm u_{s}, \quad \normmm {\p_xP(D)u}_{s}\leq C\normmm u_{s}, \quad
\normmm {\p^2_xP(D)u}_{s}\leq C\normmm u_{s},
\end{equation}
and
\begin{equation}\label{e3}
\normmm{P(D)[u,(-\p_x^2)^\n]v_x}_{s}\leq C\normmm u_{s}\normmm v_{s}.
\end{equation}
\end{lemma}
\begin{proof} The properties \eqref{5.3} and \eqref{e1} follow directly from analogous ones found in \cite{HM} just by replacing $H^2$ with $B^{s_0}_{2,1}$ and we then prove \eqref{e2}. Since
\begin{equation*}
\|\p_x^k\p^2_xP(D)u\|_{B^{s_0}_{2,1}}
\leq C\|\p_x^ku\|_{B^{s_0-2\n+2}_{2,1}}\leq C\|\p_x^ku\|_{B^{s_0}_{2,1}},
\end{equation*}
then it follows that
\begin{align*}
\normmm {\p^2_xP(D)u}_{s}
&=\sup_{k\in\N_0}\frac {s^k\|\p_x^k\p^2_xP(D)u\|_{B^{s_0}_{2,1}}}{k!/(k+1)^2}\nonumber\\
&\leq \sup_{k\in\N_0}\frac {s^k\|\p^ku\|_{B^{s_0}_{2,1}}}{k!/(k+1)^2}
=\normmm u_{s}.
\end{align*}
The other estimates in \eqref{e2} can be obtained similarly as above. Now we prove \eqref{e3}. As
\begin{align*}
&\|\p_x^kP(D)[u,(-\p_x^2)^\n]v_x\|_{B^{s_0}_{2,1}}\nonumber\\
&\leq \|\p_x^k\left([u,(-\p_x^2)^\n]v_x\right)\|_{B^{s_0-2\n}_{2,1}}\nonumber\\
&=\sum_{l=0}^{k}\left(                 
  \begin{array}{ccc}   
    k \\  
    l \\  
  \end{array}
\right)
\|[\p_x^{k-l}u,(-\p_x^2)^\n](\p_x^lv)_x\|_{B^{s_0-2\n}_{2,1}}\nonumber\\
&\leq C \sum_{l=0}^{k}\left(                 
  \begin{array}{ccc}   
    k \\  
    l \\  
  \end{array}
\right)
\|\p_x^{k-l}u\|_{B^{s_0}_{2,1}}\|\p_x^{l}v\|_{B^{s_0}_{2,1}}\nonumber\\
&=C\|\p_x^{k}u\|_{B^{s_0}_{2,1}}\|v\|_{B^{s_0}_{2,1}}
+C \sum_{l=1}^{k}\left(                 
  \begin{array}{ccc}   
    k \\  
    l \\  
  \end{array}
\right)\|\p_x^{k-l}u\|_{B^{s_0}_{2,1}}\|\p_x^{l}u\|_{B^{s_0}_{2,1}}
\nonumber\\
&\leq C \normmm v_{s}\|\p_x^{k}u\|_{B^{s_0}_{2,1}}
+C \sum_{l=1}^{k}\left(                 
  \begin{array}{ccc}   
    k \\  
    l \\  
  \end{array}
\right)\|\p_x^{k-l}u\|_{B^{s_0}_{2,1}}\|\p_x^{l}u\|_{B^{s_0}_{2,1}}.
\end{align*}
Then proceeding a similar argument as the proof of Lemma 2.1 in \cite{YY}, we complete the proof of Lemma \ref{lem5.1}.
\end{proof}
\begin{theorem}(\cite{RD5})\label{the5.1}
Let $\{X_s\}_{0<s<1}$ be a scale of decreasing Banach spaces, namely for any
$s'<s$ we have $X_{s}\subset X_{s'}$ and $\normmm {\cdot}_{s'}\leq \normmm {\cdot}_{s}$. Consider the Cauchy problem
\begin{equation}\label{5.4}
\begin{cases}
\frac{du}{dt}=F(t,u(t)),\\
u(0)=0.
\end{cases}
\end{equation}
Let $T, R$ and $C$ be positive constants and assume that $F$ satisfies the following conditions:

(1) If for $0<s'<s<1$ the function $t\mapsto u(t)$ is real analytic in $|t|< T$ and continuous on $|t|\leq T$ with values in $X_s$ and
\begin{equation}\label{5.5}
\sup_{|t|\leq T}\normmm u_s<R,
\end{equation}
then $t\mapsto F(t,u(t))$ is a real analytic function on $|t|< T$ with values in $X_{s'}$.

(2) For any $0<s'<s<1$ and any $u,v\in X_s$ with $\normmm u_s<R$, $\normmm v_s<R$,
\begin{equation}\label{5.6}
\sup_{|t|\leq T}\normmm{F(t,u)-F(t,v)}_{s'}\leq \frac {C}{s-s'}\normmm {u-v}_s.
\end{equation}

(3) There exists $M>0$ such that for any $0<s<1$
\begin{equation}\label{5.7}
\sup_{|t|\leq T}\normmm{F(t,0)}_{s}\leq \frac {M}{1-s}.
\end{equation}
Then there exists a $T_0\in (0,T)$ and a unique function $u(t)$, which for every $0<s<1$ is real analytic in $|t|<(1-s) T_0$ with values in $X_s$, and is a solution to the Cauchy problem \eqref{5.4}.
\end{theorem}
To prove Theorem \ref{the1.2}, we need to show that all three conditions of the abstract version of the Cauchy-Kowalevski theorem (Theorem \ref{the5.1}) hold for system \eqref{3.2} on the scale $\{X_s\}_{0<s<1}$. To this end, we restate the Cauchy problem \eqref{3.2} in a more convenient form. Let $u_1=u$, $u_2=u_x$, then the problem \eqref{3.2} is transformed in a system for $u_1$ and $u_2$.
\begin{equation}\label{5.8}
\begin{cases}
\p_tu_1=-u_2-\frac 1 2 \p_x(u_1^2)+\p_xP(D)f_1(u_1)+P(D)f_2(u_1,u_2)=F_1(u_1,u_2),\\
\p_tu_2=-\p_x(u_2+u_1u_2)+\p^2_xP(D)f_1(u_1)+\p_xP(D)f_2(u_1,u_2)\\
\qquad=F_2(u_1,u_2),\\
u_1(x,0)=u_0(x),u_2(x,0)=u'_0(x).
\end{cases}
\end{equation}
\begin{proof}[Proof of Theorem \ref{the1.2}] Let $u=(u_1,u_2)$, $F=(F_1,F_2)$ in \eqref{5.8} and $X_s$ be a scale of decreasing Banach spaces defined as $X_s=E_s\times E_s$. Since the map $F(u_1,u_2)$ does not depend on $t$ explicitly, we just need to verify the first two conditions of Theorem \ref{the5.1}.

Obviously, $t\mapsto F(t,u(t))=(F_1(u_1,u_2),F_2(u_1,u_2))$ is real analytic if $t\mapsto u_1(t)$ and $t\mapsto u_2(t)$ are both real analytic. Hence, the verification of the first condition of the abstract theorem needs only to show that for $s'<s$, $F_1(u_1,u_2)$ and $F_2(u_1,u_2)$ are in $E_{s'}$ for $u_1,u_2\in E_s$. By Lemma \ref{lem5.1}, we can get the estimates of $F_1$ and $F_2$ as
\begin{align}\label{5.9}
\normmm{F_1(u_1,u_2)}_{s'}&=\normmm{-u_2-\frac 1 2 \p_x(u_1^2)+\p_xP(D)f_1(u_1)+P(D)f_2(u_1,u_2)}_{s'}\nonumber\\
&\leq \frac {C}{s-s'}\LC\normmm{u_1}_s+\normmm{u_1}^2_s\RC+C\LC\normmm{u_1}_s+\normmm{u_1}^2_s\RC,
\end{align}
and
\begin{align}\label{5.10}
&\normmm{F_2(u_1,u_2)}_{s'}\nonumber\\
&=\normmm{-\p_x(u_2+u_1u_2)+\p^2_xP(D)f_1(u_1)+\p_xP(D)f_2(u_1,u_2)}_{s'}\nonumber\\
&\leq \frac {C}{s-s'}\LC\normmm{u_2}_s+\normmm{u_1}_s\normmm{u_2}_s+\normmm{u_1}^2_s\RC+C\LC\normmm{u_1}_s+\normmm{u_1}^2_s\RC.
\end{align}
We proceed to verify the second condition of the abstract theorem. Employing the triangle inequality and Lemma \ref{lem5.1}, we have
\begin{align}\label{5.11}
&\normmm{F_1(u_1,u_2)-F_1(v_1,v_2)}_{s'}\nonumber\\
&\leq\normmm{(v_2-u_2)+\frac 1 2 \p_x(v_1^2-u_1^2)}_{s'}
\nonumber\\
&+\normmm{\p_xP(D)\LC f_1(u_1)-f_1(v_1)\RC+P(D)\LC f_2(u_1,u_2)-f_2(v_1,v_2)\RC}_{s'}\nonumber\\
&\leq C\normmm{u_2-v_2}_{s}+\frac{C}{s-s'}\normmm{u_1-v_1}_{s}\normmm{u_1+v_1}_{s}\nonumber\\
&+\normmm{\p_xP(D)\LC (u_1-v_1)+(u_1-v_1)(u_1+v_1)\RC}_{s'}\nonumber\\
&+\normmm{P(D)\LC[u_1-v_1,(-\p_x^2)^{\n}]\p_xu_1-[v_1,(-\p_x^2)^{\n}](v_1-u_1)_x\RC}_{s'}\nonumber\\
&\leq C\normmm{u_2-v_2}_{s}
+\frac{C}{s-s'}\normmm{u_1-v_1}_{s}(\normmm{u_1}_{s}+\normmm{v_1}_{s})+C\normmm{u_1-v_1}_{s}\nonumber\\
&+C\normmm{u_1-v_1}_{s}(\normmm{u_1}_{s}+\normmm{v_1}_{s})\nonumber\\
&\leq \frac {C}{s-s'}\normmm{u-v}_{X_s}.
\end{align}
Similarly, we can show that
\begin{equation}\label{5.12}
\normmm{F_2(u_1,u_2)-F_2(v_1,v_2)}_{s'}\leq \frac {C}{s-s'}\normmm{u-v}_{X_s}
\end{equation}
holds. The conditions (1)-(3) are now easily verified once our system \eqref{5.8} is transformed into a new system with zero initial data as in \eqref{5.4}. This completes the proof of Theorem \ref{the1.2}.
\end{proof}

\vspace{0.5cm}
\noindent {\bf Acknowledgements.}
The work of Fan is supported by a NSFC Grant No. 11701155. The work of Gao is partially supported  by the NSFC grant No. 11531006, and the Jiangsu Center for Collaborative Innovation in Geographical Information Resource and Applications. The work of Yan is supported by the NSFC Grant No. 11771127.

\end{document}